\documentclass[12pt]{amsart}
  
\usepackage{amsmath}
\usepackage{amsthm}
\usepackage{amssymb}
\usepackage{mathrsfs}

\usepackage{booktabs}
\usepackage{siunitx}
\usepackage{multirow}
\usepackage{cite}

\usepackage[left=3cm, right=3cm, bottom=3cm]{geometry}
\usepackage{array}
\usepackage{enumerate}
\usepackage{color}
\usepackage[colorlinks=true, linkcolor=blue, citecolor=blue]{hyperref}
\numberwithin{equation}{section}
\usepackage{enumitem}
\setlist[enumerate,1]{label={\upshape(\roman*)}}

\theoremstyle{theorem}
\newtheorem{theorem}{Theorem}[section]

\newtheorem{lemma}[theorem]{Lemma}

\theoremstyle{definition}

\newtheorem{remark}[theorem]{Remark}

\theoremstyle{remark}
\theoremstyle{proof}

\newcommand{\R}{\mathbb{R}}

\newcommand{\I}{\mathrm{I}}

\renewcommand{\H}{\mathrm{H}}

\begin{document}

\title{Sharp stability for LSI}

\author{Emanuel Indrei}
\address{Department of Mathematics, Purdue University, 150 N. University Street, West Lafayette, IN 47907-2067, USA}

\date{\today}
\begin{abstract}
A solution is given to a problem discussed by Brigati, Dolbeault, and Simonov.
\end{abstract}
\maketitle

\section{Introduction}
The following quantitative LSI was proved in \cite{Log}: there exists a dimensionless $\kappa>0$ such that assuming $u \in H^1(e^{-\pi |x|^2}dx)$,  

\begin{equation} \label{r}
\pi \delta_*(u):=\int |\nabla u|^2 e^{-\pi |x|^2}dx-\pi \int |u|^2 \ln \Big(\frac{|u|^2}{||u||_{L^2}^2} \Big)  e^{-\pi |x|^2}dx \ge \kappa \inf_{a, c} \int |u-ce^{a \cdot x}|^2 e^{-\pi |x|^2}dx.
\end{equation}

In \cite[p. 5]{arXiv:2302} the stability problem relative to a stronger norm is stated:\\

``a stability on the Gaussian logarithmic Sobolev inequality is shown in [23], although the distance is measured only by an $L^2(\mathbb{R}^n, d\gamma)$ norm. Whether a stronger estimate can be obtained in the limiting case $p=2$, eventually under some restriction, is therefore so far an open question."\\

The authors suspected that unlike the $L^2$-stability \eqref{r}, a quantitative stability in a stronger norm may not hold for all of $H^1$. This is actually the case. The optimal condition to have the $H^1$ convergence is identified in  Theorem \ref{p5}. Moreover, there exists an explicit $H^1$ bound via a moment assumption. Also, \eqref{r} is sharp via Theorem \ref{@5y}. 
\pagebreak
\begin{theorem} \label{p5}
\noindent 1. Let $\{u_k\}$ be normalized and centered in $L^2(e^{-\pi |x|^2}dx)$ and suppose $\delta_*(u_k) \rightarrow 0$ as $k\rightarrow \infty$, 
then
$$|u_k|\to 1$$
in $H^1(e^{-\pi |x|^2}dx)$ if and only if

$$
m_2(u_k):=\int |x|^2|u_k(x)|^2e^{-\pi |x|^2}dx \rightarrow \int |x|^2e^{-\pi |x|^2}dx=m_2(1).\\
$$

\noindent 2. 
If $u$ is normalized and centered in $L^2(e^{-\pi |x|^2}dx)$ $\&$
$$
m_4(u):=\int |x|^4|u(x)|^2e^{-\pi |x|^2}dx \le \alpha,
$$
then
$$
\big|\big||u|-1\big|\big|_{H^1(e^{-\pi |x|^2}dx)} \le a_\alpha \Big(\delta_*^{\frac{1}{2}}(u)+\delta_*(u)\Big)^{\frac{1}{2}},
$$
$a_\alpha>0$.\\

\noindent 3. There are densities $\{u_k\}$ normalized and centered in $L^2(e^{-\pi |x|^2}dx)$, $\delta_*(u_k) \rightarrow 0$, and 
$$
||u_k||_{H^1(e^{-\pi |x|^2}dx)} \rightarrow \infty.
$$
\end{theorem}

\begin{theorem} \label{@5y}
\eqref{r} has the optimal rate.
\end{theorem}

A simple version of \eqref{r} appears in the next lemma. Note that thanks to this reduction, one may without loss of generality assume the functions to be centered and normalized in Theorem \ref{p5}.

\begin{lemma} \label{e}
\eqref{r}
is equivalent to
$$
\int |\nabla w|^2 e^{-\pi |x|^2}dx-\pi \int |w|^2 \ln |w|^2   e^{-\pi |x|^2}dx \ge \kappa  \int \big|w-1\big|^2 e^{-\pi |x|^2}dx
$$
in the space of non-negative functions which satisfy
$$
||w||_{L^2(e^{-\pi |x|^2}dx)}=1
$$
$$
\int x |w|^2e^{-\pi |x|^2}dx=0.
$$
\end{lemma}

In particular, a completely equivalent version of \eqref{r} with a moment assumption and modulus $\omega$ was already proven in \cite{IK18} utilizing a combination of optimal transport theory and Fourier analysis. Observe also that the non-negativity assumption appeared in Carlen's proof of the equality cases \cite{MR1132315}.
Suppose without loss of generality that
$$
||u||_{L^2(e^{-\pi |x|^2}dx)}=1
$$
$$
\int x |u|^2e^{-\pi |x|^2}dx=0.
$$
Now set 
$dm=2^{\frac{n}{2}} e^{-2\pi |x|^2}dx$,
$d\gamma=(2\pi)^{-\frac{n}{2}}e^{-\frac{|x|^{2}}{2}}dx$,
$$
w(x)=u(\sqrt{2}x)
$$
$$f(x)=|u|^2(\frac{x}{\sqrt{2\pi}})$$
\& observe

$$
\int |w|^2 dm=\int |u|^2e^{-\pi|x|^2}dx=\int f d\gamma
$$

$$
\int x f d\gamma = 0 = \int x |w|^2 dm.
$$
Suppose 

$$
\int |x|^2|u|^2 e^{-\pi|x|^2}dx\le M_\alpha,
$$
$M_\alpha \ge \frac{n}{2\pi}$. Note that $|\nabla |w||=|\nabla w|$ a.e., therefore assume $w \ge 0$.
An application of \cite[Corollary 1.21]{IK18} then implies that there exists a modulus $\omega$ such that
$$
\int |w-1|^2dm \le a \omega(\delta_c(w)) =a\omega(\delta_*(u))
$$
where $a=a(M_\alpha)>0$,
$$
\delta_c(w):=\frac{1}{2\pi} \int |\nabla w|^2 dm -\int |w|^2\ln|w|^2 dm.
$$
Thanks to 

$$
\int |w-1|^2dm  = \int |u-1|^2e^{-\pi|x|^2}dx,
$$

$$
\omega \Big(\int |\nabla u|^2 e^{-\pi |x|^2}dx-\pi \int |u|^2 \ln |u|^2   e^{-\pi |x|^2}dx\Big) \ge \frac{1}{\bar{a}}  \int |u-1|^2 e^{-\pi |x|^2}dx,
$$
$\bar{a}>0$. Moreover if $n=1$,
$$
\delta(f):=\frac{1}{2}\I(f)-\H(f)=\frac{1}{2}\int \frac{|\nabla f|^2}{f} d\gamma-\int f \ln f d\gamma ,
$$
\cite[Theorem 1.1]{IK18} yields
$$
\int |f-1|d\gamma \le \overline{a}_1 \delta^{\frac{1}{4}}(f),
$$
with $\overline{a}_1=\overline{a}_1(M_\alpha)>0$. Therefore assuming $u \ge 0$, 

\begin{align*}
\int |u-1|^2e^{-\pi|x|^2}dx &=\int |\sqrt{f}-1|^2 d\gamma\\
& \le \int |f-1|d\gamma\\
& \le \overline{a}_1\delta^{\frac{1}{4}}(f) = \overline{a}_1\delta_*^{\frac{1}{4}}(u).
\end{align*}
In addition, higher dimensional quantitative inequalities that also included an explicit modulus appeared in \cite{MR3271181, IK18, MR3567822}.

The optimal inequality via Theorem \ref{@5y} is
$$
||u-1||_{L^2(e^{-\pi |x|^2}dx)} \lesssim \delta_*^{\frac{1}{2}}(u),
$$
and the more general stability
\begin{equation} \label{v}
||\nabla u||_{L^2(e^{-\pi |x|^2}dx)} \le a_1 \delta_*^{\frac{1}{2}}(u)
\end{equation}
with the sharp exponent $\frac{1}{2}$ was proved in \cite{MR3271181} for probability measures which are absolutely continuous with respect to the Gaussian measure $d\gamma$ and the density $f(x)=|u|^2(\frac{x}{\sqrt{2\pi}})$ satisfies a $\log$-$C^{1,1}$ assumption. This was achieved via optimal transport theory \cite[Theorem 1.1, Remark 4.3]{MR3271181}: with the assumptions, there exists $\alpha \in (0,\frac{1}{2})$ such that
$$
\int f \ln f d\gamma \le \alpha \int \frac{|\nabla f|^2}{f} d\gamma.
$$
Hence
$$
\delta(f)=\frac{1}{2}\int \frac{|\nabla f|^2}{f} d\gamma-\int f \ln f d\gamma \ge (\frac{1}{2}-\alpha)\int \frac{|\nabla f|^2}{f} d\gamma,
$$
thus proving \eqref{v} with a simple change of variables. A surprising Ornstein-Uhlenbeck semigroup proof enables the explicit calculation of the sharp $a_1$ with a Poincar\'e assumption on $fd\gamma$, $f: \mathbb{R}^n \rightarrow \mathbb{R}_+$ \cite{MR3567822}. Observe that the Poincar\'e assumption implies the $\log$-$C^{1,1}$ assumption where one of the inequalities for the $\log$-$C^{1,1}$ assumption is precluded. In particular, it is one of the rare inequalities that highlights the sharp exponent and constant of proportionality. However, the $4^{th}$ moment assumption in Theorem \ref{p5}  includes the Poincar\'e assumption and approximates the optimal moment assumption. There exists a sequence $\{u_k\}$ where

$$\delta_*(u_k) \rightarrow 0$$

$$
\int |x|^2|u_k(x)|^2e^{-\pi |x|^2}dx \rightarrow \infty.
$$

A more general stability inequality for probability measures which are absolutely continuous with respect to the Gaussian measure was obtained with a combination of a Wasserstein metric and entropy. The techniques in \cite{MR3271181, MR3567822, MR3666798, IK18} involve optimal transport, semigroup theory, Fourier analysis, and probability. The recent proof of \eqref{r} in \cite{Log} is a fundamental achievement. One interesting feature is the lack of additional assumptions for \eqref{r} via the Bianchi-Egnell method. The first metric-stability result for LSI was obtained in \cite{MR3271181}.

Since the logarithmic Sobolev inequality has appeared in different fields, e.g. optimal transport theory, probability, statistical mechanics, quantum field theory, Riemannian geometry, thermodynamics, and information theory, there are many proofs and a large collection of articles recently investigated various stability formulations of similar inequalities: see for instance \cite{MR1132315, MR3271181, MR3567822, MR3666798, MR4455233, MR4475270, arXiv.2211, MR4368350, ARXIV.2201., MR4529870, /blms.12723, 3758731, I20, crysta ,MR3487241, 3269872, MR4116725, MR4079808, 2007.0367, 8550/ARXI}.

\section{Proofs}

\begin{proof}[Proof of Lemma \ref{e}]
Assume 
$$
||u||_{L^2(e^{-\pi |x|^2}dx)}=1
$$
$$
\int x |u|^2e^{-\pi |x|^2}dx=0;
$$
set 
$$
u=ce^{a \cdot x};
$$
then note that $a=0$, $c=1$ (via $u \ge 0$). In particular, $u=1$ is the only normalized \& centered minimizer.
Also, assuming the analog

\begin{equation} \label{salje}
\int |\nabla w|^2 e^{-\pi |x|^2}dx-\pi \int |w|^2 \ln |w|^2   e^{-\pi |x|^2}dx \ge \kappa  \int |w-1|^2 e^{-\pi |x|^2}dx,
\end{equation}
if $w\ge 0$,
$$
||w||_{L^2(e^{-\pi |x|^2}dx)}=1,
$$
$$
\int x |w|^2e^{-\pi |x|^2}dx=0,
$$
one also obtains \eqref{r}: 
let 
$$
||u||=||u||_{L^2(e^{-\pi |x|^2}dx)}
$$
$$
\alpha:=\int x|u|^2e^{-\pi |x|^2}dx
$$

$$
w(x):=\frac{u(x+\frac{\alpha}{||u||^2})e^{-\frac{\pi}{||u||^2}(\alpha \cdot x+\frac{|\alpha|^2}{2||u||^2})}}{||u||}
$$

\& observe

$$
||w||= ||w||_{L^2(e^{-\pi |x|^2}dx)}=1
$$

$$
\int x |w|^2e^{-\pi |x|^2}dx=0
$$

\begin{align}\notag
\int |w-1|^2 e^{-\pi |x|^2}dx&=\frac{1}{||u||^2}\int \Big|u(x+\frac{\alpha}{||u||^2})e^{-(\frac{\pi \alpha}{||u||^2} \cdot x+\frac{\pi|\alpha|^2}{2||u||^4})}-||u||\Big|^2 e^{-\pi |x|^2}dx\\\notag
&= \frac{1}{||u||^2} \int \Big|u(x+\frac{\alpha}{||u||^2})-e^{(\frac{\pi \alpha}{||u||^2} \cdot x+\frac{\pi|\alpha|^2}{2||u||^4}+\ln ||u|||)}\Big|^2 e^{-2(\frac{\pi \alpha}{||u||^2} \cdot x+\frac{\pi|\alpha|^2}{2||u||^4})}e^{-\pi |x|^2}dx\\\notag
&=\frac{1}{||u||^2}\int\Big |u(x+\frac{\alpha}{||u||^2})-e^{(\frac{\pi \alpha}{||u||^2} \cdot x+\frac{\pi|\alpha|^2}{2||u||^4}+\ln ||u||)}\Big|^2 e^{-\pi \big|x+\frac{\alpha}{||u||^2} \big|^2}dx\\ \label{sdjlsajw}
&=\frac{1}{||u||^2}\int\Big |u(x)-e^{(\frac{\pi \alpha}{||u||^2} \cdot x-\frac{\pi |\alpha|^2}{2||u||^4}+\ln ||u||)}\Big|^2 e^{-\pi |x|^2}dx.
\end{align}

\begin{align}\label{sdjlsaj}
&\int |\nabla w|^2 e^{-\pi |x|^2}dx\\\notag
&=\frac{1}{||u||^2} \int \Big|\nabla u(x+\frac{\alpha}{||u||^2})e^{-\frac{\pi}{||u||^2}(\alpha \cdot x+\frac{|\alpha|^2}{2||u||^2})}-\frac{\pi \alpha}{||u||^2} u(x+\frac{\alpha}{||u||^2})e^{-\frac{\pi}{||u||^2}(\alpha \cdot x+\frac{|\alpha|^2}{2||u||^2})}  \Big|^2 e^{-\pi |x|^2}dx\\\notag
&=\frac{1}{||u||^2} \Big(\int\Big |\nabla u(x+\frac{\alpha}{||u||^2})\Big|^2 e^{-\pi \big| x+  \frac{\alpha}{||u||^2}\big|^2} dx-\Big \langle\frac{2\pi \alpha}{||u||^2}, \int \nabla u(x+\frac{\alpha}{||u||^2})u(x+\frac{\alpha}{||u||^2})e^{-\pi \big| x+\frac{\alpha}{||u||^2} \big|^2}dx \Big \rangle\\\notag
&+\frac{\pi^2 |\alpha|^2}{||u||^4}\int\Big | u(x+\frac{\alpha}{||u||^2})\Big|^2 e^{-\pi \big| x+\frac{\alpha}{||u||^2}\big|^2} dx \Big)\\\notag
&=\frac{1}{||u||^2} \Big(\int |\nabla u(x)|^2 e^{-\pi | x|^2} dx-\Big \langle\frac{2\pi \alpha}{||u||^2}, \int \nabla u(x)u(x)e^{-\pi | x|^2}dx \Big \rangle+\frac{\pi^2 |\alpha|^2}{||u||^4}\int | u(x)|^2 e^{-\pi | x|^2} dx \Big).\\\notag
\end{align}
Moreover,
$$
\int \nabla u(x)u(x)e^{-\pi | x|^2}dx = -\int \nabla u(x)u(x)e^{-\pi | x|^2}dx +2 \pi\int x|u(x)|^2 e^{-\pi | x|^2}dx
$$
yields
$$
\int \nabla u(x)u(x)e^{-\pi | x|^2}dx=\pi\int x|u(x)|^2 e^{-\pi | x|^2}dx=\pi \alpha.
$$
Observe
$$
\Big \langle\frac{2\pi \alpha}{||u||^2}, \int \nabla u(x)u(x)e^{-\pi | x|^2}dx \Big \rangle= \frac{2\pi^2 |\alpha|^2}{||u||^2},
$$
therefore \eqref{sdjlsaj} implies
\begin{align} \notag
&\frac{1}{||u||^2} \Big(\int |\nabla u(x)|^2 e^{-\pi | x|^2} dx-\Big \langle\frac{2\pi \alpha}{||u||^2}, \int \nabla u(x)u(x)e^{-\pi | x|^2}dx \Big \rangle+\frac{\pi^2 |\alpha|^2}{||u||^4}\int | u(x)|^2 e^{-\pi | x|^2} dx \Big)\\ \notag
&=\frac{1}{||u||^2} \Big(\int |\nabla u(x)|^2 e^{-\pi | x|^2} dx-\frac{\pi^2 |\alpha|^2}{||u||^2} \Big)=\int |\nabla w|^2 e^{-\pi |x|^2}dx.\\ \label{sljf}
\end{align}
Analogously,
\begin{align*}
&\int |w|^2 \ln |w|^2   e^{-\pi |x|^2}dx=\frac{1}{||u||^2} \Big( \int\Big |u(x+\frac{\alpha}{||u||^2})\Big|^2\ln \Big(\frac{|u(x+\frac{\alpha}{||u||^2})|^2}{||u||^2}\Big)e^{-\pi \big| x+  \frac{\alpha}{||u||^2}\big|^2} dx\\
&-\Big \langle\frac{2\pi \alpha}{||u||^2}, \int x \Big| u(x+\frac{\alpha}{||u||^2})\Big|^2 e^{-\pi \big| x+\frac{\alpha}{||u||^2}\big|^2} dx\Big \rangle-\frac{\pi |\alpha|^2}{||u||^2}\Big);
\end{align*}

$$
 \int x\Big | u(x+\frac{\alpha}{||u||^2})\Big|^2 e^{-\pi \big| x+\frac{\alpha}{||u||^2}\big|^2} dx=\int \Big(x-\frac{\alpha}{||u||^2}\Big)|u(x)|^2 e^{-\pi | x|^2} dx=0;
$$
thus
\begin{equation} \label{lsjl}
\int |w|^2 \ln |w|^2   e^{-\pi |x|^2}dx =\frac{1}{||u||^2} \Big( \int |u(x)|^2\ln\Big(\frac{|u(x)|^2}{||u||^2}\Big)e^{-\pi | x|^2} dx-\frac{\pi |\alpha|^2}{||u||^2}\Big).
\end{equation}
Observe that \eqref{lsjl} and \eqref{sljf} imply
\begin{equation} \label{aljk}
\int |\nabla w|^2 e^{-\pi |x|^2}dx-\pi \int |w|^2 \ln |w|^2   e^{-\pi |x|^2}dx = \frac{1}{||u||^2}\Big( \int |\nabla u|^2 e^{-\pi |x|^2}dx-\pi \int |u|^2 \ln \Big(\frac{|u|^2}{||u||^2}\Big)   e^{-\pi |x|^2}dx \Big)
\end{equation}
and therefore \eqref{aljk} and \eqref{sdjlsajw} combine with \eqref{salje}:

$$
\int |\nabla u|^2 e^{-\pi |x|^2}dx-\pi \int |u|^2 \ln \Big(\frac{|u|^2}{||u||_{L^2}^2} \Big)  e^{-\pi |x|^2}dx \ge \kappa \inf_{a, c} \int |u-ce^{a \cdot x}|^2 e^{-\pi |x|^2}dx.
$$
\end{proof}

\begin{proof}[Proof of Theorem \ref{p5}]
\noindent 1. Observe that with $dm=2^{\frac{n}{2}} e^{-2\pi |x|^2}dx$ and $f \in L^2(dm)$ normalized, 
setting 
$$
f=u(\sqrt{2}x),
$$
$$
\int |f|^2 dm=\int |u|^2 e^{-\pi |x|^2}dx=1
$$
\&
$$
\frac{1}{2\pi} \int |\nabla f|^2 dm -\int |f|^2\ln|f|^2 dm=\frac{1}{\pi}\int |\nabla u|^2 e^{-\pi |x|^2}dx-\int |u|^2 \ln |u|^2   e^{-\pi |x|^2}dx.
$$
Thus \cite[Theorem 1.17]{IK18} implies
\begin{align} \label{al;sj}
&\pi \int |x|^2e^{-\pi |x|^2}dx-\pi \int |x|^2|u(x)|^2e^{-\pi |x|^2}dx +\frac{1}{\pi} \int |\nabla u(x)|^2 e^{-\pi |x|^2}dx \\ \notag
&\le \sqrt{2n} \Big(\frac{1}{\pi}\int |\nabla u|^2 e^{-\pi |x|^2}dx-\int |u|^2 \ln |u|^2   e^{-\pi |x|^2}dx \Big)^{\frac{1}{2}}\\ \notag
&+\Big( \frac{1}{\pi}\int |\nabla u|^2 e^{-\pi |x|^2}dx-\int |u|^2 \ln |u|^2   e^{-\pi |x|^2}dx\Big).\\\notag
\end{align}
Note that supposing
$$
\Big( \int |x|^2e^{-\pi |x|^2}dx- \int |x|^2|u_k(x)|^2e^{-\pi |x|^2}dx\Big) \rightarrow 0,
$$
\& $\delta_*(u_k) \rightarrow 0$,
one then obtains thanks to \eqref{al;sj} that
$$
\int |\nabla u_k(x)|^2 e^{-\pi |x|^2}dx \rightarrow 0
$$
and this implies $H^1(e^{-\pi |x|^2}dx)$ convergence.
Conversely, assuming 
$\sqrt{f}\in W^{1,2}(\R^{n},d\gamma)$ where $d\gamma=(2\pi)^{-\frac{n}{2}}e^{-\frac{|x|^{2}}{2}}dx$
denotes the Gaussian measure, the end of the proof of \cite[Proposition C.1]{IK18} implies
$$
	\frac{1}{2}\int \frac{|\nabla f|^{2}}{f}d\gamma-\int f\ln f d\gamma
	\geq \frac{1}{4n}\Big(2\int f\ln f d\gamma+(m_{2}(\gamma)-m_{2}(fd\gamma))\Big)^{2}.
$$
Thus set
$$
f_k=|u_k|^2(\frac{x}{\sqrt{2\pi}}).
$$
Then
$$
\int f_k e^{-\frac{|x|^{2}}{2}}(2\pi)^{-\frac{n}{2}}dx=     \int |u_k|^2 e^{-\pi |x|^2}dx= 1,
$$
$$
\frac{1}{2}\int \frac{|\nabla f_k|^{2}}{f_k}d\gamma-\int f_k\ln f_k d\gamma= \frac{1}{\pi}\int |\nabla u_k|^2 e^{-\pi |x|^2}dx-\int |u_k|^2 \ln |u_k|^2   e^{-\pi |x|^2}dx.
$$
In particular
\begin{align} \label{alw}
&\frac{1}{4n}\Big(2\int |u_k|^2 \ln |u_k|^2   e^{-\pi |x|^2}dx+
2\pi \int |x|^2 e^{-\pi |x|^2}dx-2\pi \int |x|^2|u_k|^2 e^{-\pi |x|^2}dx\Big)^2\\ \notag
&\le \frac{1}{\pi}\int |\nabla u_k|^2 e^{-\pi |x|^2}dx-\int |u_k|^2 \ln |u_k|^2   e^{-\pi |x|^2}dx . 
\end{align}
Observe via $H^1(e^{-\pi |x|^2}dx)$ convergence and the LSI that

$$
\int |u_k|^2 \ln |u_k|^2   e^{-\pi |x|^2}dx \rightarrow 0,
$$
thus thanks to
$$
\delta_*(u_k)=\Big(\frac{1}{\pi}\int |\nabla u_k|^2 e^{-\pi |x|^2}dx-\int |u_k|^2 \ln |u_k|^2   e^{-\pi |x|^2}dx\Big) \rightarrow 0,
$$
\eqref{alw} implies
$$
\Big(\int |x|^2 e^{-\pi |x|^2}dx- \int |x|^2|u_k|^2 e^{-\pi |x|^2}dx\Big)\rightarrow 0.\\
$$

\noindent 2.
$$
 \int |x|^4|u(x)|^2e^{-\pi |x|^2}dx\le A
$$
implies
\begin{align*}
&\Big|\pi \int |x|^2e^{-\pi |x|^2}dx-\pi \int |x|^2|u(x)|^2e^{-\pi |x|^2}dx\Big|=\pi \Big| \int |x|^2[1-|u|][1+|u|]e^{-\pi |x|^2}dx \Big| \\
& \le\pi \Big[ \int |x|^4|1+|u||^2  e^{-\pi |x|^2}dx\Big]^{1/2}\Big[   \int   |1-|u||^2e^{-\pi |x|^2}dx \Big]^{1/2} \\
& \le \pi \Big[ 2(A+\int |x|^4e^{-\pi |x|^2}dx)\Big]^{1/2}\Big[   \int   |1-|u||^2e^{-\pi |x|^2}dx \Big]^{1/2}.\\
\end{align*}
In particular \cite[Theorem 1.17]{IK18}, Lemma \ref{e} ($|u| \ge 0$, $|\nabla |u||=|\nabla u|$ a.e.), and \cite[Theorem 2]{Log} imply

\begin{align*}
&\frac{1}{\pi} \int |\nabla u(x)|^2 e^{-\pi |x|^2}dx \\
&\le \Big|\pi \int |x|^2e^{-\pi |x|^2}dx-\pi \int |x|^2|u(x)|^2e^{-\pi |x|^2}dx\Big|+\sqrt{2n} \Big(\frac{1}{\pi}\int |\nabla u|^2 e^{-\pi |x|^2}dx-\int |u|^2 \ln |u|^2   e^{-\pi |x|^2}dx \Big)^{\frac{1}{2}}\\
&+\Big( \frac{1}{\pi}\int |\nabla u|^2 e^{-\pi |x|^2}dx-\int |u|^2 \ln |u|^2   e^{-\pi |x|^2}dx\Big)\\
&\le  \pi \Big[ 2(A+\int |x|^4e^{-\pi |x|^2}dx)\Big]^{1/2}\Big[   \int   |1-|u||^2e^{-\pi |x|^2}dx \Big]^{1/2}+\\
&\sqrt{2n} \Big(\frac{1}{\pi}\int |\nabla u|^2 e^{-\pi |x|^2}dx-\int |u|^2 \ln |u|^2   e^{-\pi |x|^2}dx \Big)^{\frac{1}{2}}+\Big( \frac{1}{\pi}\int |\nabla u|^2 e^{-\pi |x|^2}dx-\int |u|^2 \ln |u|^2   e^{-\pi |x|^2}dx\Big)\\
&\le \pi \Big[ 2(A+\int |x|^4e^{-\pi |x|^2}dx)\Big]^{1/2} \sqrt{\frac{\pi}{\kappa}}\Big[  \frac{1}{\pi}\int |\nabla u|^2 e^{-\pi |x|^2}dx-\int |u|^2 \ln |u|^2   e^{-\pi |x|^2}dx \Big]^{1/2}+\\
&\sqrt{2n} \Big(\frac{1}{\pi}\int |\nabla u|^2 e^{-\pi |x|^2}dx-\int |u|^2 \ln |u|^2   e^{-\pi |x|^2}dx \Big)^{\frac{1}{2}}+\Big( \frac{1}{\pi}\int |\nabla u|^2 e^{-\pi |x|^2}dx-\int |u|^2 \ln |u|^2   e^{-\pi |x|^2}dx\Big).
\end{align*}
%

\noindent 3.
Suppose $fd\gamma$ is a probability measure \& let $T=\nabla \Phi$ be the Brenier map between $fd\gamma$ and $d\gamma$. Then it follows from the proof of LSI via optimal transport \cite{MR1894593} that 
$$\int |T(x)-x+\nabla \ln f|^2 fd\gamma \le 2 \delta(f).$$ 
The argument is as follows: define $\theta:= \Phi-\frac{1}{2}|x|^2$. Note that 
$$
f(x) e^{-|x|^2/2}=\text{det}(I+D^2 \theta(x)) e^{-|x+\nabla \theta(x) |^2/2}.
$$
Next, taking the logarithm and then integrating:

\begin{align*}
\int f \ln f d\gamma &\le \int f \Big[\Delta \theta - x \cdot \nabla \theta \Big] d\gamma-\frac{1}{2} \int |\nabla \theta|^2 f d\gamma\\
&=-\int \nabla \theta \cdot \nabla f d\gamma-\frac{1}{2} \int |\nabla \theta|^2 f d\gamma\\
&=-\frac{1}{2}\int \Big|\nabla \theta+\frac{\nabla f}{f} \Big|^2 fd\gamma +  \frac{1}{2} \int \frac{|\nabla f|^2}{f}  d\gamma.\\
\end{align*}
Therefore
 $$
 \frac{1}{2} \int \frac{|\nabla f|^2}{f}  d\gamma-\int f \ln f d\gamma \ge  \frac{1}{2}\int \Big|T-x+\nabla \ln f \Big|^2 fd\gamma.
 $$
Thus, Jensen's inequality yields
$$
2 \delta(f) \ge \int |T(x)-x+\nabla \ln f|^2 fd\gamma \ge \Big(\int |T(x)-x+\nabla \ln f| fd\gamma \Big)^2; 
$$ 
in particular

$$
 \int |T(x)-x+\nabla \ln f| fd\gamma \le \sqrt{2 \delta(f)}.
$$ 
Observe $T$ is the Brenier map, nevertheless supposing one considers the Monge-cost, it yields an upper bound on $W_1$ (in the one-dimensional case, the inequality is an equality):
$$
W_1(f d\gamma,\gamma) \le \int |T(x)-x| fd\gamma.
$$
This then implies
\begin{align*}
W_1(f d\gamma,\gamma) &\le \sqrt{2 \delta(f)}+\int |\nabla f|d\gamma\\
&\le \sqrt{2 \delta(f)}+\Big(\int \frac{|\nabla f|^2}{f} d\gamma \Big)^{\frac{1}{2}}.
\end{align*}
Therefore let $\{f_k d\gamma\}$ be a sequence of probability measures with 
$$
0<\liminf_k W_1(f_k d\gamma,\gamma)
$$
$$
\delta(f_k)\rightarrow 0;
$$
thus
$$
0<\liminf_k \int \frac{|\nabla f_k|^2}{f_k} d\gamma.
$$
Set
$$
f_k=|u_k|^2(\frac{x}{\sqrt{2\pi}}).
$$
One then has
$$
\int f_k e^{-\frac{|x|^{2}}{2}}(2\pi)^{-\frac{n}{2}}dx=\int |u_k|^2 e^{-\pi |x|^2}dx=1
$$
$$
\int \frac{|\nabla f_k|^{2}}{f_k}d\gamma =\frac{2}{\pi}\int |\nabla u_k|^2 e^{-\pi |x|^2}dx.
$$
Therefore

$$
0<\liminf_k \int |\nabla u_k|^2 e^{-\pi |x|^2}dx
$$
and this directly implies

$$
||u_k||_{H^1(e^{-\pi |x|^2}dx)} \nrightarrow 0
$$
$$
\delta(u_k) \rightarrow 0.
$$
Furthermore, 
there is a sequence of probability measures such that
$$
W_1(f_k d\gamma,d\gamma) \rightarrow \infty, 
$$
see \cite{MR4455233}, hence the argument above finishes the proof.

\end{proof}

\begin{remark}
Observe that $H^1$ convergence is equivalent to $W_2$ convergence (when $\delta(u_k) \rightarrow 0$). The first $W_2$ bound was obtained in \cite{MR3271181}. 
\end{remark}

\begin{proof}[Proof of Theorem \ref{@5y}]
If $a>0$, define
$$
u_a(x)=(2a+1)^{\frac{n}{4}}e^{-a\pi|x|^2}.
$$
Now 
$$
||u_a||_{L^2(e^{-\pi |x|^2}dx)}=1,
$$
$$
\int x |u_a|^2e^{-\pi |x|^2}dx=0,
$$

\begin{align*}
&\frac{\Big(\int |\nabla u|^2 e^{-\pi |x|^2}dx-\pi \int |u|^2 \ln |u|^2   e^{-\pi |x|^2}dx\Big)}{\int |u-1|^2 e^{-\pi |x|^2}dx}\\
&=\pi\Big(\frac{\frac{2na^2}{2a+1}-\frac{n}{2}\ln(2a+1)+\frac{na}{2a+1}}{2-2\frac{(2a+1)^{\frac{n}{4}}}{(a+1)^{\frac{n}{2}}}}\Big)\\
&=\frac{\pi}{2} \frac{2na^2-\frac{n}{2}(2a+1)\ln(2a+1)+na}{2a+1-\frac{(2a+1)^{\frac{n+4}{4}}}{(a+1)^{\frac{n}{2}}}},
\end{align*}

\begin{align*}
&\lim_{a \rightarrow 0^+} \frac{\Big(\int |\nabla u|^2 e^{-\pi |x|^2}dx-\pi \int |u|^2 \ln |u|^2   e^{-\pi |x|^2}dx\Big)}{\int |u-1|^2 e^{-\pi |x|^2}dx}\\
&\lim_{a \rightarrow 0^+} \frac{\pi}{2} \frac{4na-n\ln(2a+1)}{2-\Big (\frac{n+4}{2}\frac{(2a+1)^{\frac{n}{4}}}{(a+1)^{\frac{n}{2}}} - \frac{n}{2}\frac{(2a+1)^{\frac{n+4}{4}}}{(a+1)^{\frac{n}{2}+1}}\Big)}\\
&\lim_{a \rightarrow 0^+} -\pi\frac{2n-n\frac{1}{2a+1}}{\Big ( \Big(  \frac{(a+1)^{\frac{n}{2}}\frac{n}{2}\frac{n+4}{2} (2a+1)^{\frac{n-4}{4}}-(2a+1)^{\frac{n}{4}}\frac{n+4}{2}\frac{n}{2}(a+1)^{\frac{n-2}{2}}}{(a+1)^{n}}\Big)
-\Big (\frac{(a+1)^{\frac{n}{2}+1}\frac{n}{2}\frac{n+4}{2}(2a+1)^{\frac{n}{4}} -(2a+1)^{\frac{n+4}{4}}\frac{n}{2}(a+1)^{\frac{n}{2}}(\frac{n}{2}+1)}{(a+1)^{n+2}}\Big)\Big)}\\
&=2\pi.
\end{align*}
In particular, assume via contradiction 
$$
\omega \Big(\int |\nabla u|^2 e^{-\pi |x|^2}dx-\pi \int |u|^2 \ln |u|^2   e^{-\pi |x|^2}dx\Big) \ge \kappa  \int |u-1|^2 e^{-\pi |x|^2}dx
$$
when $u$ is normalized and centered, with
$$
\omega(a)=o(a).
$$
Hence
\begin{align*}
&\kappa \le \frac{\omega \Big(\int |\nabla u_a|^2 e^{-\pi |x|^2}dx-\pi \int |u_a|^2 \ln |u_a|^2   e^{-\pi |x|^2}dx\Big)}{\int |u_a-1|^2 e^{-\pi |x|^2}dx} \\
&=\frac{\omega \Big(\int |\nabla u_a|^2 e^{-\pi |x|^2}dx-\pi \int |u_a|^2 \ln |u_a|^2   e^{-\pi |x|^2}dx\Big)}{\int |\nabla u_a|^2 e^{-\pi |x|^2}dx-\pi \int |u_a|^2 \ln |u_a|^2   e^{-\pi |x|^2}dx} \frac{\int |\nabla u_a|^2 e^{-\pi |x|^2}dx-\pi \int |u_a|^2 \ln |u_a|^2   e^{-\pi |x|^2}dx}{\int |u_a-1|^2 e^{-\pi |x|^2}dx}\\
& \rightarrow 0 
\end{align*}
as $a \rightarrow 0^+$; this therefore contradicts $\kappa>0$. 
\end{proof}

{\bf Acknowledgment}
I worked on some of the content during the minisymposium “Qualitative properties of solutions to Elliptic and Parabolic PDEs and related
topics"  in Gaeta, Italy (May 2019). The excellent academic environment is acknowledged. 
\bibliography{References}
\bibliographystyle{amsplain}

\end{document}